\documentclass[final,3p,review]{elsarticle}

\usepackage{amssymb}
\usepackage{amsmath}
\usepackage{upgreek}
\usepackage{dsfont}
\usepackage{tikz}
\usetikzlibrary{calc}
\usepackage{hyperref}
\hypersetup{bookmarksopen=true, bookmarksnumbered=true}
\usepackage{enumitem}
\setlist[enumerate]{itemsep=0pt, parsep=0pt, topsep=0pt, label=\upshape(\arabic*)}
\usepackage{extarrows}
\usepackage[all]{xy}
\allowdisplaybreaks[4]
\usepackage{upref}
\usepackage{caption}
\usepackage{amsthm}
\newtheorem{theorem}{Theorem}[section]
\newtheorem*{mainthm}{Theorem \ref{thm-main}}
\newtheorem{proposition}[theorem]{Proposition}
\newtheorem{lemma}[theorem]{Lemma}
\newtheorem{corollary}[theorem]{Corollary}
\newtheorem{claim}[theorem]{Claim}
\newdefinition{definition}[theorem]{Definition}
\newdefinition{remark}[theorem]{Remark}
\newdefinition{example}[theorem]{Example}
\newdefinition{question}[theorem]{Question}

\numberwithin{equation}{section}

\newcommand\bF{\mathbb{F}}
\newcommand\bN{\mathbb{N}}

\newcommand\rd{\operatorname{d}}

\newcommand\Id{\mathrm{Id}}

\newcommand\abs[1]{\lvert #1 \rvert}
\newcommand\xto[1][\sim]{\xrightarrow{#1}}
\newcommand\inv{^{-1}}
\DeclareMathOperator{\rank}{rank}
\DeclareMathOperator{\im}{Im}
\newcommand{\GL}{\operatorname{GL}}

\usepackage{lineno}

\journal{arXiv}

\date{September 13, 2023}

\begin{document}

\begin{frontmatter}

\title{On exterior powers of reflection representations}

\author{Hongsheng Hu}
\ead{huhongsheng16@mails.ucas.ac.cn}
\address{Beijing International Center for Mathematical Research, Peking University, \\ No. 5 Yiheyuan Road, Haidian District, Beijing 100871, China}

\begin{abstract}

In 1968, R. Steinberg proved a theorem stating that the exterior powers of an irreducible reflection representation of a Euclidean reflection group are again irreducible and pairwise non-isomorphic.
We extend this result to a more general context where the inner product invariant under the group action may not necessarily exist.
\end{abstract}


\end{frontmatter}

\section{Introduction} \label{sec-intro}

In \cite[\S 14.1, \S 14.3]{Steinberg1968}, R. Steinberg proved the following theorem (see also \cite[Ch. V, \S 2, Exercise 3.]{Bourbaki2002}, \cite[Theorem 9.13]{CIK71}, \cite[Theorem 5.1.4]{GP00} and \cite[\S 24-3]{Kane01}).

\begin{theorem}  [R. Steinberg] \label{thm-steinberg}
  Let $V$ be a finite-dimensional vector space endowed with an inner product (for example, a Euclidean space, or a complex Hilbert space).
  Let $\{v_1, \dots, v_n\}$ be a basis of $V$, and $W \subseteq \GL(V)$ be the group generated by (orthogonal) reflections with respect to these basis vectors.
  Suppose $V$ is a simple $W$-module.
  Then the $W$-modules $\{\bigwedge^d V \mid 0 \le d \le n\}$ are simple and pairwise non-isomorphic.
\end{theorem}

The proof relies on the existence of an inner product which stays invariant under the $W$-action.
With the help of this inner product, the vector space $\bigwedge^d V$ is decomposed into a direct sum $\bigwedge^d V^\prime \bigoplus (v \wedge \bigwedge^{d-1} V^\prime)$, where $V^\prime$ is a subspace of $V$ of codimension one and a simple module of a subgroup generated by fewer reflections, and $v$ is a vector orthogonal to $V^\prime$.
The theorem is proved by induction on the number of reflections.

In this paper, we extend this result to a more general context, where the $W$-invariant inner product may not exist.
The following is the main theorem.

\begin{theorem} \label{thm-main}
  Let $\rho:W \to \GL(V)$ be an $n$-dimensional representation of a group $W$ over a field $\bF$ of characteristic 0.
  Suppose that $s_1, \dots, s_k \in W$ satisfy:
  \begin{enumerate}
    \item for each $i$, $s_i$ acts on $V$ by a (generalized) reflection with reflection vector $\alpha_i$ of eigenvalue $\lambda_i$ (see Definition \ref{def-refl} for related notions);
    \item the group $W$ is generated by $\{s_1, \dots, s_k\}$;
    \item the representation $(V,\rho)$ is irreducible;
    \item \label{thm-main-4} for any pair $i,j$ of indices, $s_i \cdot \alpha_j \ne \alpha_j$ if and only if $s_j \cdot \alpha_i \ne \alpha_i$.
  \end{enumerate}
  Then the $W$-modules $\{\bigwedge^d V \mid 0 \le d \le n\}$ are irreducible and pairwise non-isomorphic.
\end{theorem}

\begin{remark} \leavevmode
  \begin{enumerate}
    \item The condition \ref{thm-main-4} in Theorem \ref{thm-main} is a technical condition (automatically satisfied in the setting of Theorem \ref{thm-steinberg}).
        However, it is not that strict.
        For example, if $s_i$ and $s_j$ are both of order $2$ (so that they generate a dihedral subgroup), and if $s_i \cdot \alpha_j = \alpha_j$ while $s_j \cdot \alpha_i \ne \alpha_i$, then the order of $s_is_j$ in $W$ must be $\infty$.
        Moreover, (note that $\langle \alpha_i,\alpha_j \rangle$ forms a subrepresentation of the subgroup $\langle s_i, s_j \rangle$) there are uncountably many two-dimensional representations of the infinite dihedral group $\langle s_i, s_j \rangle$, but only two of them invalidate the condition \ref{thm-main-4} (see \cite[\S 2.2]{Hu23}).
    \item In the paper \cite{Hu23}, we construct and classify a class of representations of an arbitrary Coxeter group of finite rank, where the defining generators of the group act by (generalized) reflections.
        In view of the previous remark, most of these reflection representations satisfy the conditions of Theorem \ref{thm-main}.
        Thus, our result applies to them giving many irreducible representations of the Coxeter group.
        (Note that in \cite{Hu23} we have seen that only a few reflection representations admit a nonzero bilinear form which is invariant under the group action.
        Consequently, the module $\bigwedge^d V$ usually fails to decompose into the form $\bigwedge^d V^\prime \bigoplus (v \wedge \bigwedge^{d-1} V^\prime)$ as it did in proving Theorem \ref{thm-steinberg}.
        Even if we have such a decomposition, the subspace $V^\prime$ may not be a simple module of a suitable subgroup.
        Therefore, the arguments in proving Theorem \ref{thm-steinberg} usually fail in the context of Theorem \ref{thm-main}.)
  \end{enumerate}
\end{remark}

This paper is organized as follows.
In Sections \ref{sec-refl} - \ref{sec-graph}, we revisit basic concepts and provide the background we need concerning (generalized) reflections, exterior powers, and graphs.
In Section \ref{sec-main}, we prove our main theorem.
In Section \ref{sec-other}, we present some byproducts, including a description of the subspace of an exterior power that is fixed pointwise by a set of reflections, and a Poincar\'e-like duality on exterior powers.
In Section \ref{sec-question}, we raise several interesting questions that have yet to be resolved.

\section{Generalized reflections} \label{sec-refl}

Let $\bF$ be a field and $V$ be a finite-dimensional vector space  over $\bF$.

\begin{definition} \label{def-refl} \leavevmode
  \begin{enumerate}
    \item A linear map $s:V \to V$ is called a  (\emph{generalized}) \emph{reflection} if $s$ is diagonalizable and $\rank(s - \Id_V) = 1$.
    \item Suppose $s$ is a reflection on $V$.
        The hyperplane $H_s : = \ker (s - \Id_V)$, which is fixed pointwise by $s$, is called the \emph{reflection hyperplane} of $s$.
        Let $\alpha_s$ be a nonzero vector in $\im(s - \Id_V)$.
        Then, $s \cdot \alpha_s = \lambda_s \alpha_s$ for some $\lambda_s \in \bF \setminus \{1\}$, and $\alpha_s$ is called a \emph{reflection vector} of $s$.
  \end{enumerate}
\end{definition}

Note that if $s$ is an invertible map, then $\lambda_s \ne 0$.

\begin{lemma} \label{lem-refl}
  Let $s$ be a reflection on $V$.
  Then there exists a nonzero linear function $f : V \to \bF$ such that $s \cdot v = v + f(v) \alpha_s$ for any $v \in V$.
\end{lemma}

\begin{proof}
  Note that $V = H_s \bigoplus \bF \alpha_s$.
  Any vector $v$ can be written in the form $v = v_s + c_v \alpha_s$ where $v_s \in H_s$ and $c_v \in \bF$.
  Then $$s \cdot v = v_s + \lambda_s c_v \alpha_s = v + (\lambda_s - 1) c_v \alpha_s.$$
  The linear function $f: v \mapsto (\lambda_s - 1) c_v$ is the desired function.
\end{proof}

\section{Exterior powers} \label{sec-ext}

In this section, let $W$ be a group and $\rho: W \to \GL(V)$ be a representation of $W$, where $V$ is an $n$-dimensional vector space over the base field $\bF$.
Let $\bigwedge^d V$ ($0 \le d\le n$) be the $d$-th exterior power of $V$.
The representation $\bigwedge^d \rho$ of $W$ on $\bigwedge^d V$ is given by
$$w \cdot (v_1 \wedge \dots \wedge v_d) = (w \cdot v_1) \wedge \dots \wedge (w \cdot v_d).$$
By convention, $\bigwedge^0 V$ is the one-dimensional $W$-module with trivial action.
On the other hand, $\bigwedge^n V$ carries the one-dimensional representation $\det \circ \rho$.
For more details, one may refer to \cite{FH91}.

The following is a well-known fact.

\begin{lemma} \label{lem-wedge-basis}
  Suppose $\{\alpha_1, \dots, \alpha_n\}$ is a basis of $V$.
  Then,
  $$\{ \alpha_{i_1} \wedge \dots \wedge \alpha_{i_d} \mid 1 \le i_1 < \dots < i_d \le n\}$$
  is a basis of $\bigwedge^d V$ $(0 \le d \le n)$.
  In particular, $\dim \bigwedge^d V = \binom{n}{d}$.
\end{lemma}

Suppose an element $s$ of $W$ acts on $V$ by a reflection, with reflection hyperplane $H_s$ and reflection vector  $\alpha_s$ of eigenvalue $\lambda_s$ (see Definition \ref{def-refl}).
Note that $W$ is a group and $s$ is invertible.
Thus, $\lambda_s \ne 0$.
We define
\begin{equation}  \label{eq-Vds}
  V_{d,s}^+ = \Bigl\{v \in \bigwedge^d V \Bigm\vert s \cdot v = v\Bigr\}, \quad V_{d,s}^- = \Bigl\{v \in \bigwedge^d V \Bigm\vert s \cdot v = \lambda_s v\Bigr\}
\end{equation}
to be the eigen-subspaces of $s$ in $\bigwedge^d V$, for the eigenvalues $1$ and $\lambda_s$, respectively.

\begin{lemma} [See also {\cite[Lemma 5.1.2]{GP00}} and the proof of {\cite[Proposition 9.12]{CIK71}}] \label{lem-basis-1}
  Let $W, s, V$ be as above.
  Suppose $\{v_1, \dots, v_{n-1}\}$ is a basis of $H_s$.
  Then $V_{d,s}^+$ $(0 \le d \le n)$ has a basis
  \begin{equation}   \label{eq-lem-basis-1}
    \{ v_{i_1} \wedge \dots \wedge v_{i_d} \mid 1 \le i_1 < \dots < i_d \le n-1 \},
  \end{equation}
  and $V_{d,s}^-$ has a basis
  \begin{equation}   \label{eq-lem-basis-2}
    \{ \alpha_s \wedge v_{i_1} \wedge \dots \wedge v_{i_{d-1}} \mid 1 \le i_1 < \dots < i_{d-1} \le n-1 \}.
  \end{equation}
  In particular, $\dim V_{d,s}^+ = \binom{n-1}{d}$, $\dim V_{d,s}^- = \binom{n-1}{d-1}$, and $\bigwedge^d V =  V_{d,s}^+ \bigoplus V_{d,s}^-$.
  (Here we regard $\binom{n-1}{n} = \binom{n-1}{-1} = 0$.)
\end{lemma}

\begin{proof}
  Note that $\{\alpha_s, v_1, \dots, v_{n-1}\}$ is a basis of $V$.
  Denote by $B^+$ and $B^-$ the two sets of vectors in \eqref{eq-lem-basis-1} and \eqref{eq-lem-basis-2}, respectively.
  Then the disjoint union $B^+ \cup B^-$ is a basis of $\bigwedge^d V$ by Lemma \ref{lem-wedge-basis}.
  Clearly, $B^+ \subseteq V_{d,s}^+$ and $B^- \subseteq V_{d,s}^-$.
  Therefore $\bigwedge^d V =  V_{d,s}^+ \bigoplus V_{d,s}^-$ and the result follows.
\end{proof}

\begin{corollary} \label{cor-eigen-basis}
  Let $W, s, V$ be as above.
  \begin{enumerate}
    \item \label{cor-eigen-basis-1} We have $V_{d,s}^+ = \bigwedge^d H_s$. Here $\bigwedge^d H_s$ is regarded as a subspace of $\bigwedge^d V$ naturally.
    \item \label{cor-eigen-basis-2} Extend $\alpha_s$ arbitrarily to a basis of $V$, say, $\{\alpha_s, \alpha_2, \dots, \alpha_n\}$.
         Then  $V_{d,s}^-$ $(0 \le d \le n)$ has a basis
         \begin{equation}   \label{eq-basis-3}
            \{ \alpha_s \wedge \alpha_{i_1} \wedge \dots \wedge \alpha_{i_{d-1}} \mid 2 \le i_1 < \dots < i_{d-1} \le n \}.
         \end{equation}
  \end{enumerate}
\end{corollary}

\begin{proof}
  The point \ref{cor-eigen-basis-1} is directly derived  from Lemma \ref{lem-basis-1}.
  For \ref{cor-eigen-basis-2}, suppose $s \cdot \alpha_i = \alpha_i + c_i \alpha_s$ ($c_i \in \bF$) for $i = 2, \dots, n$ (see Lemma \ref{lem-refl}).
  Then,
  \begin{align*}
    s \cdot (\alpha_s \wedge \alpha_{i_1} \wedge \dots \wedge \alpha_{i_{d-1}}) & = (\lambda_s \alpha_s) \wedge (\alpha_{i_1} + c_{i_1} \alpha_s) \wedge \dots \wedge (\alpha_{i_{d-1}} + c_{i_{d-1}} \alpha_s) \\
    & = \lambda_s \alpha_s \wedge \alpha_{i_1} \wedge \dots \wedge \alpha_{i_{d-1}}.
  \end{align*}
  Thus, $\alpha_s \wedge \alpha_{i_1} \wedge \dots \wedge \alpha_{i_{d-1}} \in V_{d,s}^-$.
  Note that by Lemma \ref{lem-wedge-basis}, the $\binom{n-1}{d-1}$ vectors in \eqref{eq-basis-3} are linearly independent, and that $\dim V_{d,s}^- = \binom{n-1}{d-1}$ by Lemma \ref{lem-basis-1}.
  Thus, the vectors in \eqref{eq-basis-3} form a basis of $V_{d,s}^-$.
\end{proof}

For a subset $B \subset V$, we denote by $\langle B \rangle$ the linear subspace spanned by $B$.
By convention, $\langle  \emptyset \rangle = 0$.
The following lemma is evident by linear algebra.

\begin{lemma} \label{lem-cap} \leavevmode
  Let $B \subseteq V$ be a basis of $V$, and $B_i \subseteq B$ $(i \in I)$ be a family of subsets of $B$.
  Then, $\bigcap_{i \in I} \langle B_i \rangle = \langle \bigcap_{i \in I} B_i \rangle$.
\end{lemma}

Using the lemmas above, we can deduce the following results which will be used in the proof of our main theorem.

\begin{proposition} \label{prop-key}
  Let $W, V$ be as above.
  Suppose $s_1, \dots, s_k \in W$ such that, for each $i$, $s_i$ acts on $V$ by a reflection with reflection vector  $\alpha_i$ of eigenvalue $\lambda_i$.
  Suppose $\alpha_1, \dots, \alpha_k$ are linearly independent.
  We extend these vectors to a basis of $V$, say, $\{\alpha_1, \dots, \alpha_k, \alpha_{k+1}, \dots, \alpha_n\}$.
  \begin{enumerate}
    \item If $0 \le d < k$, then $\bigcap_{1 \le i\le k} V_{d, s_i}^- = 0$.
    \item \label{prop-key-2} If $k \le d \le n$, then $\bigcap_{1 \le i\le k} V_{d, s_i}^-$ has a basis
         \begin{equation*}
           \{\alpha_1 \wedge \dots \wedge \alpha_k \wedge \alpha_{j_{k+1}} \wedge \dots \wedge \alpha_{j_d} \mid k+1 \le j_{k+1} < \dots < j_d \le n\},
         \end{equation*}
         In particular, if $d = k$, then $\bigcap_{1 \le i\le k} V_{k, s_i}^-$ is one-dimensional with a basis vector $\alpha_1 \wedge \dots \wedge \alpha_k$.
  \end{enumerate}
\end{proposition}

\begin{proof}
  By Corollary \ref{cor-eigen-basis}\ref{cor-eigen-basis-2}, $V_{d,s_i}^-$ ($1 \le i \le k$) has a basis
  $$B_i := \{\alpha_{j_1} \wedge \dots \wedge \alpha_{j_d} \mid 1 \le j_1 < \dots < j_d \le n, \text{ and $j_l = i$ for some $l$}\}.$$
  Note that by Lemma \ref{lem-wedge-basis}, the ambient space $\bigwedge^d V$ has a basis
  $$B : = \{\alpha_{j_1} \wedge \dots \wedge \alpha_{j_d} \mid 1 \le j_1 < \dots < j_d \le n\}$$
  and $B_i \subseteq B$, for all $i = 1, \dots, k$.
  By Lemma \ref{lem-cap}, we have
  $$\bigcap_{i=1}^{k} V_{d,s_i}^- = \bigcap_{i=1}^{k} \langle B_i \rangle = \Bigl\langle \bigcap_{i=1}^k B_i \Bigr\rangle.$$
  If $0 \le d < k$, then $\bigcap_{1 \le i \le k} B_i = \emptyset$ and $\bigcap_{1 \le i\le k} V_{d, s_i}^- = 0$.
  If $k \le d \le n$, then
  $$\bigcap_{i=1}^k B_i = \{\alpha_1 \wedge \dots \wedge \alpha_k \wedge \alpha_{j_{k+1}} \wedge \dots \wedge \alpha_{j_d} \mid k+1 \le j_{k+1} < \dots < j_d \le n\}$$
  and $\bigcap_{1 \le i \le k} B_i$ is a basis of $\bigcap_{1 \le i\le k} V_{d, s_i}^-$.
\end{proof}

\begin{proposition} [See also the proofs of {\cite[Theorem 9.13]{CIK71} and \cite[Theorem 5.14]{GP00}}] \label{prop-k=l}
  Let $W, V$ be as above.
  Suppose there exists $s \in W$ such that $s$ acts on $V$ by a reflection.
  If $0 \le k,l \le n$ are integers and $\bigwedge^k V \simeq \bigwedge^l V$ as $W$-modules, then $k = l$.
\end{proposition}

\begin{proof}
  If $\bigwedge^k V \simeq \bigwedge^l V$, then $\dim \bigwedge^k V = \dim \bigwedge^l V$ and $\dim V_{k,s}^+ = \dim V_{l,s}^+$.
  By Lemmas \ref{lem-wedge-basis} and \ref{lem-basis-1}, this is equivalent to
  \begin{equation*}
    \binom{n}{k} = \binom{n}{l}, \quad\quad \binom{n-1}{k} = \binom{n-1}{l}.
  \end{equation*}
  The two equalities force $k = l$.
\end{proof}

The following lemma is due to C. Chevalley \cite[page 88]{Chevalley1955} (see also \cite[Corollary 22.45]{Milne2017}).

\begin{lemma}[C. Chevalley]
Let $\bF$ be a field of characteristic 0.
Let $W$ be a group and $V$, $U$ be finite-dimensional semisimple $W$-modules over $\bF$.
Then $V \bigotimes U$ is a semisimple $W$-module.
\end{lemma}

Note that the $W$-module $\bigwedge^d V$ can be regarded as a submodule of $\bigotimes^d V$ via the natural embedding
$$\bigwedge^d V \hookrightarrow \bigotimes^d V, \quad\quad v_1 \wedge \dots \wedge v_d \mapsto \sum_{\sigma \in \mathfrak{S}_d} \operatorname{sgn}(\sigma) v_{\sigma(1)} \otimes \dots \otimes v_{\sigma(d)}.$$
(The notation $\mathfrak{S}_d$ denotes the symmetric group on $d$ elements.)
Therefore, we have the following corollary.

\begin{corollary} \label{cor-semisimple}
  Let $\bF$ be a field of characteristic 0.
  Let $W$ be a group and $V$ be a finite-dimensional simple $W$-module over $\bF$.
  Then the $W$-module $\bigwedge^d V$ is semisimple.
\end{corollary}

\section{Some lemmas on graphs} \label{sec-graph}

By definition, an \emph{(undirected) graph} $G = (S,E)$ consists of a set $S$ of vertices and a set $E$ of edges.
Each edge in $E$ is an unordered binary subset $\{s,t\}$ of $S$.
For our purpose, we only consider finite graphs without loops and multiple edges (that is, $S$ is a finite set, there is no edge of the form $\{s,s\}$, and each pair $\{s,t\}$ occurs at most once in $E$).

A sequence $(s_1, s_2, \dots, s_n)$ of vertices is called a \emph{path} in $G$ if $\{s_i, s_{i+1}\} \in E$, for all $i$.
In this case, we say that the two vertices $s_1$ and $s_n$ are \emph{connected} by the path.
A graph $G$ is called \emph{connected} if any two vertices are connected by a path.

\begin{definition} \label{def-subgraph}
  Let $G = (S,E)$ be a graph and $I \subseteq S$ be a subset.
  We set $E(I) := \{\{s,t\} \in E \mid s,t \in I\}$ to be the set of edges with vertices in $I$, and call the graph $G(I) : = (I, E(I))$ the \emph{subgraph of $G$ spanned by $I$}.
\end{definition}

\begin{definition} \label{def-move}
  Let $G = (S,E)$ be a graph and $I \subseteq S$ be a subset.
  Suppose there exists vertices $r \in I$ and $t\in S \setminus I$ such that $\{r,t\} \in E$ is an edge.
  Let $I^\prime : = (I \setminus \{r\} ) \cup \{t\}$.
  Then we say $I^\prime$ is obtained from $I$ by a \emph{move}.

  Intuitively, $I^\prime$ is obtained from $I$ by moving the vertex $r$ to the vertex $t$ along the edge $\{r,t\}$.
  In particular, $\abs{I} = \abs{I^\prime}$.
\end{definition}

We shall need the following lemmas in the proof of our main theorem.

\begin{lemma} \label{lem-move}
  Let $G = (S, E)$ be a connected graph.
  Let $I, J \subseteq S$ be subsets with cardinality $\abs{I} = \abs{J} = d$.
  Then $J$ can be obtained from $I$ by finite steps of moves.
\end{lemma}

\begin{proof}
  We do induction downwards on $\abs{I \cap J}$.
  If $\abs{I \cap J} =d$, then $I=J$ and there is nothing to prove.

  If $I \ne J$, then there exist vertices $r \in I \setminus J$ and $t \in J \setminus I$.
  Since $G$ is connected, there is a path connecting $r$ and $t$, say,
  $$(r = r_0, r_1,r_2, \dots, r_l = t).$$
  Let $0 = i_0 < i_1 < \dots < i_k < l$ be the indices such that $\{r_{i_0}, r_{i_1}, \dots, r_{i_k} \}$ is the set of vertices in $I$ on this path, that is, $\{r_{i_0}, \dots, r_{i_k} \} = \{r_i \mid 0 \le i < l, r_i \in I\}$.

  Clearly $r_{i_k+1}, r_{i_k+2}, \dots, r_l \notin I$.
  So beginning with $I$, we can move $r_{i_k}$ to $r_{i_k+1}$, then to $r_{i_k+2}$, and finally to $t$.
  Therefore, the set $I_1 := (I \setminus \{r_{i_k}\}) \cup \{t\}$ can be  obtained from $I$ by finite steps of moves.
  Similarly, from $I_1$ we can move $r_{i_{k-1}}$ to $r_{i_k}$ so that we obtain
  $$I_2 : = (I_1 \setminus \{r_{i_{k-1}}\}) \cup \{r_{i_k}\} = (I \setminus \{r_{i_{k-1}}\}) \cup \{t\}.$$
  Do this recursively, and finally we get $I_{k+1} := (I \setminus \{r_{i_{0}}\}) \cup \{t\} = (I \setminus \{r\}) \cup \{t\}$ from $I$ by finite steps of moves.

  Moreover, we have $\abs{I_{k+1} \cap J} = \abs{I \cap J} + 1$.
  By the induction hypothesis, $J$ can be obtained from $I_{k+1}$ by finite steps of moves.
  It follows that $J$ can be obtained from $I$ by finite steps of moves.
\end{proof}

\begin{lemma} \label{lem-graph-delete}
  Let $G = (S, E)$ be a connected graph.
  Suppose $I \subseteq S$ is a subset such that $\abs{I} \ge 2$, and for any $t \in S \setminus I$, either one of the following conditions is satisfied:
  \begin{enumerate}
    \item for any $r \in I$, $\{r,t\}$ is not an edge;
    \item there exist at least two vertices $r,r^\prime \in I$ such that $\{r,t\}, \{r^\prime,t\} \in E$.
  \end{enumerate}
  Then there exists $s \in I$ such that the subgraph $G(S \setminus \{s\})$ is connected.
\end{lemma}

\begin{proof}
  For two vertices $r,t \in S$, define the distance $\rd(r,t)$  in $G$ to be
  \begin{align*}
    \rd(r,t) := \min \{m \in \bN  & \mid \exists r_1, \dots, r_{m-1}\in S \text{ such that }   (r, r_1, \dots, r_{m-1}, t) \text{ is a path in } G\}.
  \end{align*}
  Let $m = \max\{\rd(r,t) \mid r,t \in I\}$, and suppose $s, s^\prime \in I$ such that $\rd(s, s^\prime) = m$.

  We claim first that
  \begin{equation}   \label{eq-claim-graph-delete-1}
    \text{for any $r \in I \setminus\{s\}$, $r$ is connected to $s^\prime$ in the subgraph $G(S \setminus \{s\})$.}
  \end{equation}
  Otherwise, any path in $G$ connecting $r$ and $s^\prime$ (note that $G$ is connected) must pass through $s$.
  It follows that  $\rd(r, s^\prime) > \rd(s, s^\prime) = m$, which contradicts our choice of $m$.

  Next we claim that
  \begin{equation}   \label{eq-claim-graph-delete-2}
    \text{for any $t \in S \setminus \{s\}$, $t$ is connected to $s^\prime$ in the subgraph $G(S \setminus \{s\})$,}
  \end{equation}
  and therefore, $G(S \setminus \{s\})$ is connected.

  In \eqref{eq-claim-graph-delete-2}, the case where $t \in I \setminus \{s\}$ has been settled in the claim \eqref{eq-claim-graph-delete-1}.
  Thus, we may assume $t \in S \setminus I$.
  Let $(t = t_0, t_1, \dots, t_p = s)$ be a path of minimal length in $G$ connecting $t$ and $s$.
  Then $t_1, \dots, t_{p-1} \in S \setminus \{s\}$.
  In particular, $t$ is connected to $t_{p-1}$ in $G(S \setminus \{s\})$.
  If $t_{p-1} \in I$, then by the claim \eqref{eq-claim-graph-delete-1}, $t_{p-1}$ is connected to $s^\prime$ in $G(S \setminus \{s\})$.
  Thus, $t$ is connected to $s^\prime$ in $G(S \setminus \{s\})$ as desired.
  If $t_{p-1} \notin I$, then, since $\{t_{p-1}, s\} \in E$, there is another vertex $r \in I \setminus \{s\}$ such that $\{t_{p-1},r\} \in E$.
  By the claim \eqref{eq-claim-graph-delete-1} again, $r$ is connected to $s^\prime$ in $G(S \setminus \{s\})$, and so is $t$.
  See Figure \ref{figure-1} for an illustration.
  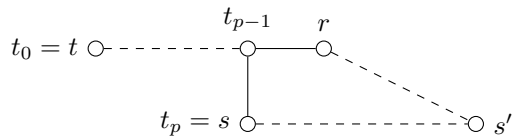
\begin{figure}[ht]
    \centering
    \begin{tikzpicture}
      \node [circle, draw, inner sep=2pt, label=left:${t_p = s}$] (s) at (0,0) {};
      \node [circle, draw, inner sep=2pt, label=right:$s^\prime$] (sp) at (3,0) {};
      \draw [dashed] (s) -- (sp);
      \node [circle, draw, inner sep=2pt, label=above:$t_{p-1}$] (tpm1) at (0,1) {};
      \node [circle, draw, inner sep=2pt, label=above:$r$] (r) at (1,1) {};
      \draw (s) -- (tpm1) -- (r);
      \node [circle, draw, inner sep=2pt, label=left:${t_0 = t}$] (t) at (-2,1) {};
      \draw [dashed] (t) -- (tpm1);
      \draw [dashed] (r) -- (sp);
    \end{tikzpicture}
    \caption{Illustration for the proof of Lemma \ref{lem-graph-delete}.}\label{figure-1}
  \end{figure}
\end{proof}

\begin{remark}
  Apply Lemma \ref{lem-graph-delete} to the trivial case $I = S$, we recover the following simple fact (see also the hint of {\cite[Ch. V, \S 2, Exercise 3(d)]{Bourbaki2002}}):
  if  $G = (S,E)$ is a connected graph, then there exists $s \in S$ such that $G(S \setminus \{s\})$ is connected.
\end{remark}

\section{Proof of Theorem \ref{thm-main}} \label{sec-main}

This section is devoted to proving Theorem \ref{thm-main}.
For the sake of readers' convenience, we restate the theorem here.

\begin{mainthm}
  Let $\rho:W \to \GL(V)$ be an $n$-dimensional representation of a group $W$ over a field $\bF$ of characteristic 0.
  Suppose that $s_1, \dots, s_k \in W$ satisfy:
  \begin{enumerate}
    \item for each $i$, $s_i$ acts on $V$ by a (generalized) reflection with reflection vector $\alpha_i$ of eigenvalue $\lambda_i$;
    \item the group $W$ is generated by $\{s_1, \dots, s_k\}$;
    \item the representation $(V,\rho)$ is irreducible;
    \item for any pair $i,j$ of indices, $s_i \cdot \alpha_j \ne \alpha_j$ if and only if $s_j \cdot \alpha_i \ne \alpha_i$.
  \end{enumerate}
  Then the $W$-modules $\{\bigwedge^d V \mid 0 \le d \le n\}$ are irreducible and pairwise non-isomorphic.
\end{mainthm}

By Proposition \ref{prop-k=l}, we see that the $W$-modules $\{\bigwedge^d V \mid 0 \le d \le n\}$ are pairwise non-isomorphic.
Thus, to prove Theorem \ref{thm-main}, it suffices to show that $\bigwedge^d V$ is a simple $W$-module for each fixed $d$.

By Corollary \ref{cor-semisimple}, the $W$-module $\bigwedge^d V$ is semisimple.
Therefore, the problem reduces to proving
\begin{equation}   \label{eq-main}
  \text{any endomorphism of $\bigwedge^d V$ is a scalar multiplication.}
\end{equation}

Let $S := \{1, \dots, k\}$ and $E := \{\{i, j \} \mid s_i \cdot \alpha_j \ne \alpha_j\}$.
Then $G = (S, E)$ is a graph in the sense of Section \ref{sec-graph}.

\begin{claim} \label{claim-1}
  $G$ is a connected graph.
\end{claim}

\begin{proof}
  Otherwise, suppose $S = I \sqcup J$ such that for any $i \in I$ and $j \in J$, $\{i, j\}$ is never an edge, that is, $s_j \cdot \alpha_i = \alpha_i$.
  If $V = V_I :=\langle \alpha_i \mid i \in I \rangle$, then for any $j \in J$, $\alpha_j$ is a linear combination of $\{\alpha_i \mid i \in I\}$.
  It follows that $s_j \cdot \alpha_j = \alpha_j$, which is absurd.
  Therefore, $V_I \ne V$.

  By Lemma \ref{lem-refl}, $s_i \cdot V_I \subseteq V_I$ for any $i \in I$.
  However, for any $j \in J$, $s_j$ acts trivially on $V_I$.
  Since $W$ is generated by $\{s_1, \dots, s_k\}$, $V_I$ is closed under the action of $W$, that is, $V_I$ is a proper submodule.
  This contradicts the assumption that $V$ is a simple $W$-module.
\end{proof}

\begin{claim} \label{claim-2}
  $V$ is spanned by $\{\alpha_1, \dots, \alpha_k\}$.
  In particular, $n \le k$ (where $n = \dim V$).
\end{claim}

\begin{proof}
  Let $U = \langle \alpha_1, \dots, \alpha_k \rangle \subseteq V$.
  By Lemma \ref{lem-refl}, $s_i \cdot U \subseteq U$ for any $i \in S$.
  Thus, $U$ is a $W$-submodule.
  However, $V$ is a simple $W$-module.
  So $U = V$.
\end{proof}

\begin{claim} \label{claim-I}
  There exists a subset $I \subseteq S$, such that:
  \begin{enumerate}
    \item $\{\alpha_i \mid i \in I\}$ is a basis of $V$;
    \item the subgraph $G(I)$ (see Definition \ref{def-subgraph}) is connected.
  \end{enumerate}
\end{claim}

\begin{proof}
  Suppose we have found a subset $J \subseteq S$ such that
  \begin{enumerate} [label = (\alph*)]
    \item \label{claim-pf-1} $V$ is spanned by $\{\alpha_i \mid i \in J\}$;
    \item \label{claim-pf-2} the subgraph $G(J)$ is connected.
  \end{enumerate}
  For example, $S$ itself is such a subset by Claims \ref{claim-1} and \ref{claim-2}.
  If the vectors $\{\alpha_i \mid i\in J\}$ are linearly independent, then we are done.

  Now suppose $\{\alpha_i \mid i\in J\}$ are linearly dependent.
  By a permutation of indices, we may assume $J = \{1, \dots, h\}$, $h \le k$, and
  \begin{equation}   \label{eq-assume}
    c_1 \alpha_1 + \dots + c_l \alpha_l = 0 \quad \text{for some $c_1, \dots, c_l \in \bF^\times$, $l \le h$.}
  \end{equation}
  If there exists $j \in J$ such that $j \ge l+1$ and $s_j \cdot \alpha_i \ne \alpha_i$ for some $i \le l$, then
  $$s_j \cdot (c_1 \alpha_1 + \dots + \widehat{c_i \alpha_i} + \dots + c_l \alpha_l) \ne c_1 \alpha_1 + \dots + \widehat{c_i \alpha_i} + \dots + c_l \alpha_l.$$
  Here $\widehat{c_i \alpha_i}$ means this term is omitted.
  Thus, there is an index $i^\prime$ with $i^\prime \le l$ and $i^\prime \ne i$ such that $s_j \cdot \alpha_{i^\prime} \ne \alpha_{i^\prime}$.
  In other words, if $l+1 \le j \le h$, then either one of the following is satisfied:
  \begin{enumerate}[label = (\roman*)]
    \item for any $i \le l$, $\{i,j\}$ is never an edge;
    \item there exist at least two indices $i, i^\prime \le l$ such that $\{i,j\}, \{i^\prime, j\} \in E$.
  \end{enumerate}
  Applying Lemma \ref{lem-graph-delete} to the subset $\{1, \dots, l\} \subseteq J$,  we see that there is an index $i_0 \le l$ such that the subgraph $G(J \setminus \{i_0\})$ is connected.
  Moreover, $V$ is spanned by $\{\alpha_i \mid i \in J \setminus \{i_0\}\}$ by our assumption \eqref{eq-assume}.
  Thus, $J \setminus \{i_0\}$ satisfies the conditions \ref{claim-pf-1} and \ref{claim-pf-2}, and $J \setminus \{i_0\}$ has a smaller cardinality than $J$.

  Apply the arguments above recursively.
  Finally, we will obtain a subset $I \subseteq S$ as claimed.
\end{proof}

Now suppose $I = \{1, \dots, n\} \subseteq S$ $(n = \dim V)$ is the subset obtained in Claim \ref{claim-I}.
By Lemma \ref{lem-wedge-basis}, $\bigwedge^d V$ has a basis
$$\{ \alpha_{i_1} \wedge \dots \wedge \alpha_{i_d} \mid 1 \le i_1 < \dots < i_d \le n\}.$$
Note that for any such basis vector $\alpha_{i_1} \wedge \dots \wedge \alpha_{i_d}$, the vectors $\alpha_{i_1}, \dots, \alpha_{i_d}$ of $V$ are linearly independent.

\begin{claim} \label{claim-4}
  For any indices $1 \le i_1 < \dots < i_d \le n$, the subspace
  $\bigcap_{1 \le j \le d} V_{d, s_{i_j}}^-$ is one-dimensional with a basis vector $\alpha_{i_1} \wedge \dots \wedge \alpha_{i_d}$ (the subspace $V_{d,s_i}^-$ is defined in \eqref{eq-Vds}).
\end{claim}

\begin{proof}
  Apply Proposition \ref{prop-key}\ref{prop-key-2} to $s_{i_1}, \dots, s_{i_d} \in W$.
\end{proof}

Now suppose $\varphi: \bigwedge^d V \to \bigwedge^d V$ is an endomorphism of the $W$-module.
For any $i \in I$  and any $v \in V_{d,s_i}^-$, we have
$$s_i \cdot \varphi (v) = \varphi(s_i \cdot v) = \lambda_i \varphi(v).$$
Thus, $\varphi(v) \in V_{d,s_i}^-$.
Therefore, for any indices $1 \le i_1 < \dots < i_d \le n$, we have $\varphi (\bigcap_{1 \le j \le d} V_{d, s_{i_j}}^-) \subseteq \bigcap_{1 \le j \le d} V_{d, s_{i_j}}^-$.
By Claim \ref{claim-4}, it holds that
$$\varphi (\alpha_{i_1} \wedge \dots \wedge \alpha_{i_d}) = \gamma_{i_1, \dots, i_d} \cdot \alpha_{i_1} \wedge \dots \wedge \alpha_{i_d} \text{ for some } \gamma_{i_1, \dots, i_d} \in \bF.$$
To prove the statement \eqref{eq-main}, it suffices to show that the coefficients $\gamma_{i_1, \dots, i_d}$ are constant among all choices of $i_1, \dots, i_d$.
We may assume $d \le n-1$.

\begin{claim} \label{claim-5}
  Let $I_1 = \{1 \le i_1 < \dots < i_d \le n\}, I_2 = \{1 \le j_1 < \dots < j_d \le n\}$ be two subsets of $I$.
  Suppose $I_2$ can be obtained form $I_1$ by a move (see Definition \ref{def-move}) in the graph $G(I)$.
  Then, $\gamma_{i_1, \dots, i_d} = \gamma_{j_1, \dots, j_d}$.
\end{claim}

\begin{proof}
  To simplify notation, we assume $I_1 = \{1, \dots, d\}$, $I_2 = (I_1 \setminus \{d\}) \cup \{d+1\}$ and $\{d, d+1\} \in E$ is an edge.
  In view of Lemma \ref{lem-refl}, for $i = 1, \dots, d$, we assume
  $$s_{d+1} \cdot \alpha_i = \alpha_i + c_i \alpha_{d+1}, \quad c_i \in \bF.$$
  Then $c_d \ne 0$.
  We have
  \begin{align*}
    s_{d+1} \cdot (\alpha_1 \wedge \dots \wedge \alpha_d) & = (\alpha_1 + c_1 \alpha_{d+1}) \wedge \dots \wedge (\alpha_d + c_d \alpha_{d+1})\\
     & = \alpha_1 \wedge \dots \wedge \alpha_d   + \sum_{i=1}^{d} (-1)^{d-i} c_i \cdot \alpha_1 \wedge \dots \wedge \widehat{\alpha}_i \wedge \dots \wedge \alpha_d \wedge \alpha_{d+1}.
  \end{align*}
  Hence,
  \begin{align*}
    \varphi \bigl(s_{d+1} \cdot (\alpha_1 \wedge \dots \wedge \alpha_d)\bigr) & = \varphi \bigl(\alpha_1 \wedge \dots \wedge \alpha_d  + \sum_{i=1}^{d} (-1)^{d-i} c_i \cdot \alpha_1 \wedge \dots \wedge \widehat{\alpha}_i \wedge \dots \wedge  \alpha_{d+1}\bigr) \\
     & = \gamma_{1,\dots, d} \cdot \alpha_1 \wedge \dots \wedge \alpha_d  + \sum_{i=1}^{d} (-1)^{d-i} c_i \gamma_{1,\dots,\widehat{i},\dots, d+1} \cdot \alpha_1 \wedge \dots \wedge \widehat{\alpha}_i \wedge \dots \wedge  \alpha_{d+1} \\
     \intertext{and also equals}
     s_{d+1} \cdot \varphi(\alpha_1 \wedge \dots \wedge \alpha_d) & = \gamma_{1,\dots, d} s_{d+1} \cdot (\alpha_1 \wedge \dots \wedge \alpha_d) \\
     & = \gamma_{1,\dots, d} \cdot \alpha_1 \wedge \dots \wedge \alpha_d + \sum_{i=1}^{d} (-1)^{d-i} c_i \gamma_{1,\dots, d} \cdot \alpha_1 \wedge \dots \wedge \widehat{\alpha}_i \wedge \dots \wedge  \alpha_{d+1}.
  \end{align*}
  Note that $c_d \ne 0$, and that the vectors involved in the equation above are linearly independent.
  Thus, we have $\gamma_{1,\dots, d} = \gamma_{1,\dots, d-1, d+1}$ which is what we want.
\end{proof}

Now apply lemma \ref{lem-move} to the connected graph $G(I)$.
Then by Claim \ref{claim-5} we see that the coefficients $\gamma_{i_1, \dots, i_d}$ are constant among all choices of $i_1, \dots, i_d \in I$.
As we have pointed out, this means that the statement \eqref{eq-main} is valid.

The proof is completed.

\section{Some other results} \label{sec-other}

\begin{lemma} \label{lem-cap-wedge}
  Let $H_1, \dots, H_k \subseteq V$ be linear subspaces of a vector space $V$.
  Regard $\bigwedge^d H_i$ as a subspace of $\bigwedge^d V$ for $0 \le d \le n$.
  Then $\bigcap_{1 \le i \le k} (\bigwedge^d H_i) = \bigwedge^d (\bigcap_{1 \le i \le k} H_i)$.
\end{lemma}

\begin{proof}
  We do induction on $k$ and begin with the case $k = 2$.
  Let $I_0$ be a basis of $H_1 \cap H_2$.
  Extend $I_0$ to a basis of $H_1$, say, $I_0 \sqcup I_1$, and to a basis of $H_2$, say, $I_0 \sqcup I_2$.
  Then $I_0 \sqcup I_1 \sqcup I_2$ is a basis of $H_1 + H_2$.
  Further, extend $I_0 \sqcup I_1 \sqcup I_2$ to a basis of $V$, say, $I_0 \sqcup I_1 \sqcup I_2 \sqcup I_3$.

  We define a total order $\le$ on the set of vectors $I_0 \sqcup I_1 \sqcup I_2 \sqcup I_3$, and we write $v_1 < v_2$ if $v_1 \le v_2$ and $v_1 \ne v_2$.
  By Lemma \ref{lem-wedge-basis}, $B$, $B_1$, $B_2$, $B_0$ are bases of $\bigwedge^d V$, $\bigwedge^d H_1$, $\bigwedge^d H_2$, $\bigwedge^d (H_1 \cap H_2)$, respectively, where
  \begin{align*}
    B & := \{v_1 \wedge \dots \wedge v_d \mid v_1, \dots, v_d \in I_0 \sqcup I_1 \sqcup I_2 \sqcup I_3, \text{ and } v_1 < \dots < v_d\}, \\
    B_1 & := \{v_1 \wedge \dots \wedge v_d \mid v_1, \dots, v_d \in I_0 \sqcup I_1, \text{ and } v_1 < \dots < v_d\}, \\
    B_2 & := \{v_1 \wedge \dots \wedge v_d \mid v_1, \dots, v_d \in I_0  \sqcup I_2, \text{ and } v_1 < \dots < v_d\}, \\
    B_0 & := \{v_1 \wedge \dots \wedge v_d \mid v_1, \dots, v_d \in I_0, \text{ and } v_1 < \dots < v_d\}.
  \end{align*}
  (The sets $B_1, B_2, B_0$ may be empty.)
  Moreover, $B_1, B_2 \subseteq B$, and $B_0 = B_1 \cap B_2$.
  Apply Lemma \ref{lem-cap} to the vector space $\bigwedge^d V$.
  We obtain
  $$\Bigl(\bigwedge^d H_1\Bigr) \cap \Bigl(\bigwedge^d H_2\Bigr) =  \langle B_1 \rangle \cap \langle B_2 \rangle = \langle B_1 \cap B_2 \rangle  = \langle B_0 \rangle = \bigwedge^d (H_1 \cap H_2).$$

  For $k \ge 3$, by the induction hypothesis, we have
  \begin{align*}
    \bigcap_{i=1}^{k} \Bigl(\bigwedge^d H_i\Bigr) & = \Bigl( \bigcap_{i=1}^{k-1} \bigl(\bigwedge^d H_i\bigr) \Bigr) \cap \Bigl(\bigwedge^d H_k\Bigr) \\
     & = \Bigl( \bigwedge^d \bigl(\bigcap_{i=1}^{k-1} H_i\bigr) \Bigr) \cap \Bigl(\bigwedge^d H_k\Bigr) \\
     & = \bigwedge^d \Bigl( \bigl(\bigcap_{i=1}^{k-1} H_i\bigr) \cap H_k \Bigr) \\
     & = \bigwedge^d \Bigl(\bigcap_{i=1}^{k} H_i\Bigr)
  \end{align*}
  as desired.
\end{proof}

The following proposition, which is derived from Lemma \ref{lem-cap-wedge}, recovers \cite[\S 14.2]{Steinberg1968} in a more general context.

\begin{proposition}
  Let $\rho:W \to \GL(V)$ be a finite-dimensional representation of a group $W$.
  Suppose $s_1, \dots, s_k \in W$ such that, for each $i$, $s_i$ acts on $V$ by a reflection with reflection hyperplane $H_i$.
  Then, $\{v \in \bigwedge^d V \mid s_i \cdot v = v \text{ for all } i\} = \bigcap_{1 \le i \le k} V_{d,s_i}^+ = \bigwedge^d (\bigcap_{1 \le i \le k} H_i)$ for $0 \le d \le n$.
\end{proposition}

\begin{proof}
  By Corollary \ref{cor-eigen-basis}\ref{cor-eigen-basis-1}, $V_{d,s_i}^+ = \bigwedge^d H_i$ for all $i$.
  Therefore,
  \begin{equation*}
    \Bigl\{v \in \bigwedge^d V \Bigm\vert s_i \cdot v = v \text{ for all } i \Bigr\} = \bigcap_{i=1}^{k} V_{d,s_i}^+ = \bigcap_{i=1}^{k} \Bigl(\bigwedge^d H_i\Bigr) = \bigwedge^d \Bigl(\bigcap_{i=1}^{k} H_i\Bigr).
  \end{equation*}
  The last equality follows from Lemma \ref{lem-cap-wedge}.
\end{proof}

The next result is a Poincar\'e-like duality on exterior powers of a representation.

\begin{proposition}
  Let $\rho: W \to \GL(V)$ be an $n$-dimensional representation of a group $W$.
  Then, $\bigwedge^{n-d} V \simeq (\bigwedge^d V)^* \bigotimes (\det \circ \rho)$ as $W$-modules for all $d = 0,1, \dots, n$.
  Here we denote by $(\bigwedge^d V)^*$ the dual representation of $\bigwedge^d V$.
\end{proposition}

\begin{proof}
  Fix an identification of linear spaces $\bigwedge^n V \simeq \bF$.
  For any $d = 0, 1, \dots, n$, we define a bilinear map
  \begin{align*}
    f: \Bigl(\bigwedge^{n-d} V\Bigr) \times \Bigl(\bigwedge^{d} V\Bigr) & \to \bigwedge^{n} V \simeq \bF, \\
    (u,v) & \mapsto u \wedge v.
  \end{align*}
  Clearly, this induces an isomorphism of linear spaces
  \begin{align*}
    \varphi: \bigwedge^{n-d} V & \xto \Bigl(\bigwedge^{d} V\Bigr)^*, \\
    u & \mapsto f(u,-).
  \end{align*}
  Note that for any $w \in W$, $u \in \bigwedge^{n-d} V$, $v \in \bigwedge^{d} V$, we have
  $$f(w \cdot u, w \cdot v) = (w\cdot u) \wedge (w \cdot v) = w \cdot (u \wedge v) = \det(\rho(w)) u \wedge v = \det(\rho(w)) f(u,v).$$
  Therefore, for  $w \in W$, $u \in \bigwedge^{n-d} V$,
  \begin{equation*}
    \varphi(w \cdot u) = f(w\cdot u, -) = f(w\cdot u, (w w \inv \cdot -)) = \det (\rho(w))f(u, (w\inv \cdot -)).
  \end{equation*}
  This implies $u \mapsto f(u,-) \otimes 1$ is an isomorphism  $\bigwedge^{n-d} V \xto (\bigwedge^d V)^* \bigotimes (\det \circ \rho)$  of $W$-modules.
\end{proof}

\section{Further questions} \label{sec-question}

\begin{question}
  Can we remove the technical condition \ref{thm-main-4} in the statement of Theorem \ref{thm-main}?
\end{question}

\begin{question}
  Is it possible to find two non-isomorphic simple $W$-modules $V_1, V_2$ satisfying the conditions of Theorem \ref{thm-main}, and two integers $d_1, d_2$ with $0 < d_i < \dim V_i$, such that $\bigwedge^{d_1} V_1 \simeq \bigwedge^{d_2} V_2$ as $W$-modules?
\end{question}

If the answer to this question is negative for reflection representations of Coxeter groups, then the irreducible representations obtained in the way described in Section \ref{sec-intro} are non-isomorphic to each other.

\begin{question}
  What kinds of simple $W$-module $V$ have the property that the modules $\bigwedge^d V$, $0 \le d \le \dim V$, are simple and pairwise non-isomorphic?
  Can we formulate any other sufficient conditions (in addition to Theorem \ref{thm-main}) or any necessary conditions?
\end{question}

\section*{Acknowledgement}

The author is deeply grateful to Professor Ming Fang and Tao Gui for enlightening discussions.
The author would also like to thank the anonymous referee for suggestions, which helped to improve this paper.

\pdfbookmark[1]{References}{References}
\bibliographystyle{plain}
\bibliography{exterior-powers}

\end{document}